\documentclass{amsart}
\usepackage{url}
\usepackage{algorithm2e}
\usepackage{algorithmic}
\usepackage{hyperref}
\hypersetup{colorlinks,%
citecolor=blue,%
linkcolor=blue,%
urlcolor=blue,%
}

\renewcommand{\algorithmicrequire}{\textbf{Input:}}
\renewcommand{\algorithmicensure}{\textbf{Output:}}

\def \C {\mathbb{C}}
\def \Q {\mathbb{Q}}
\def \ni {\noindent}

\def \ms {\medskip}
\def \mni {\ms\ni}

\def \I {\mathcal{I}}
\def \A {\mathcal{A}}
\def \X {\mathcal{X}}
\def \Y {\mathcal{Y}}

\def \S {\mathcal{S}}

\def \K {\mathcal{K}}
\def \P {\mathcal{P}}
\def \ci {\mathcal{I}}
\def \cj {\mathcal{J}}

\def \exp {{\rm exp}}
\def \deg {{\rm deg}}

\def \jac {{\rm Jac}}

\def\Maple{{\sffamily\small Maple\/ }}

\newtheorem{theorem}{Theorem}[section]
\newtheorem{lemma}[theorem]{Lemma}

\newtheorem{definition}[theorem]{Definition}

\newtheorem{proposition}[theorem]{Proposition}

\numberwithin{equation}{section}

\begin{document}

\title[Locally nilpotent derivations in dimension three]
{Characterization of rank two locally nilpotent derivations in
dimension three}

\author[M. A. Barkatou]{Moulay A. Barkatou}
\address{Laboratoire XLIM, UMR 6172 CNRS-Universit\'e de
Limoges\\
Avenue Albert-Thomas 123, 87060, Limoges Cedex, France}
\email{moulay.barkatou@unilim.fr}

\author[H. El Houari]{Hassan El Houari}
\address{Department of Mathematics\\Faculty of Sciences Semlalia\\
Cadi Ayyad University\\P.O Box 2390, Marrakech\\Morocco}
\email{h.elhouari@ucam.ac.ma}

\author[M. El Kahoui]{M'hammed El Kahoui}
\address{Department of Mathematics\\Faculty of Sciences Semlalia\\
Cadi Ayyad University\\P.O Box 2390, Marrakech\\Morocco}
\email{elkahoui@ucam.ac.ma}

\date{}

\dedicatory{}



\begin{abstract}In this paper we give an algorithmic
characterization of rank two locally nilpotent derivations in
dimension three. Together with an algorithm for computing the
plinth ideal, this gives a method for computing the rank of a
locally nilpotent derivation in dimension three.
\end{abstract}

\maketitle


\section{Introduction}\label{sec:intro}
Let $\K$ be a commutative field of characteristic zero, $\K^{[n]}$
be the ring of polynomials in $n$ variables with coefficients in
$\K$ and $Aut_{\K}(\K^{[n]})$ be the group of $\K$-automorphisms
of $\K^{[n]}$. The structure of $Aut_{\K}(\K^{[2]})$ is well
understood \cite{kulk53}. However, $Aut_{\K}(\K^{[n]})$ remains a
big mystery for $n\geq 3$.

In order to understand the nature of $Aut_{\K}(\K^{[n]})$ it is
natural to investigate algebraic group actions on the affine
$n$-space over $\K$. Actions of the algebraic group $(\K,+)$ on
affine spaces are commonly called {\it algebraic $G_a$-actions},
and they are all of the form $\exp(t\X)_{t\in \K}$ where $\X$ is a
locally nilpotent $\K$-derivation of the polynomial ring
$\K^{[n]}$.

Locally nilpotent derivations of $\K^{[2]}$ are completely
classified, and this classification is algorithmic
\cite{rentschler68}. In dimension three, D. Daigle, G. Freudenburg
and S. Kaliman obtained several deep results which constitute a
big step towards a classification, see \cite{Freudenburg_book} and
the references therein. However, some of these results which are
obtained by using topological methods are not of algorithmic
nature. It would of course be very nice to obtain an algorithmic
classification of locally nilpotent derivations in dimension
three, but this seems to be a difficult problem. This paper
addresses the less ambitious problem of computing some invariants,
namely the plinth ideal and the rank, of a locally nilpotent
derivation in dimension three.

The paper is structured as follows. In section \ref{sec:basics} we
set up notation and give the main results to be used. Section
\ref{sec:local_slice} concerns minimal local slices and how to
compute them in dimension $3$. This gives an algorithm for
computing a generator of the plinth ideal of a locally nilpotent
derivation in dimension three. In section \ref{sec:rank} we give
an algorithm to compute the rank of a locally nilpotent
$\K$-derivation in dimension $3$.

\section{Notation and basic facts}\label{sec:basics}
Throughout this paper $\K$ is a commutative field of
characteristic zero, all the considered rings are commutative of
characteristic zero with unit and all the considered derivations
are nonzero. A derivation of a $\K$-algebra $\A$ is called a
$\K$-derivation if it satisfies $\X(a)=0$ for any $a\in \K$.

\subsection{Classical results on locally nilpotent derivations}
A derivation $\X$ of a ring $\A$ is called {\it locally nilpotent}
if for any $a\in\A$ there exists a positive integer $n$ such that
$\X^n(a)=0$. An element $s$ of $\A$ satisfying $\X(s)\neq 0$ and
$\X^2(s)=0$ is called a {\it local slice} of $\X$. If moreover
$\X(s)=1$ then $s$ is called a {\it slice} of $\X$. A nonzero
locally nilpotent derivation needs not to have a slice but always
has a local slice.

The following result, which dates back at least to
\cite{wright81a}, concerns locally nilpotent derivations having a
slice.
\begin{lemma}\label{wright}Let $\A$ be a ring containing $\Q$
and $\X$ be a locally nilpotent derivation of $\A$ having a slice
$s$. Then $\A=\A^{\X}[s]$ and $\X={\partial}_s$.
\end{lemma}

Let $\A$ be a ring and $\X$ be a derivation of $\A$. The subset
$\{a\in \A;\;\X(a)=0\}$ of $\A$ is in fact a subring called the
{\it ring of constants} of $\X$ and is denoted by $\A^{\X}$. When
$\A$ is a domain and $\X$ is locally nilpotent, the ring of
constants $\A^{\X}$ is factorially closed in $\A$, i.e., if $a\in
\A^{\X}$ and $a=bc$ then $b,c\in \A^{\X}$. Also, the fact that
$\A$ is of characteristic zero implies that $\A^{\star}\subset
\A^{\X}$.

Locally nilpotent derivations in two variables over fields of
characteristic $0$ are well understood. We have in particular the
following version of Rentschler's theorem \cite{rentschler68}.

\begin{theorem}\label{rentshcler}Let $\X$ be a locally nilpotent
$\K$-derivation of $\K[x,y]$. Then there exist two polynomials
$f,g$ of $\K[x,y]$ and a univariate polynomial $h$ such that
$\K[f,g]=\K[x,y]$, $\K[x,y]^{\X}=\K[f]$ and $\X=h(f){\partial}_g$.
\end{theorem}

As a consequence of theorem \ref{rentshcler}, if $\A$ is a UFD
containing $\Q$ and $\X$ is a locally nilpotent $\A$-derivation of
$\A[x,y]$ then there exists $f\in \A[x,y]$ and a univariate
polynomial $h$ such that $\A[x,y]^{\X}=\A[f]$ and
$\X=h(f)({\partial}_yf{\partial}_x-{\partial}_xf{\partial}_y)$,
see \cite{daigle98a,elkahoui2004a}.

In the case of $\K^{[3]}$ we have the following result proved by
Miyanishi \cite{Miyanishi85a} for the case $\K=\C$ and may be
extended to the general case in a straightforward way by using
Kambayashi's transfer principle \cite{kambayashi75a}, see also
\cite{makar_limanov2005} for an algebraic proof.
\begin{theorem}\label{miyanishi}Let $\X$ be a locally nilpotent
$\K$-derivation of $\K[x,y,z]$. Then there exist polynomials $f,g$
such that $\K[x,y,z]^{\X}=\K[f,g]$.
\end{theorem}

Contrary to lemma \ref{wright} and theorem \ref{rentshcler} which
are of algorithmic nature, it is not clear from the existing
proofs of theorem \ref{miyanishi} how to compute, for a given
locally nilpotent derivation $\X$ of $\K[x,y,z]$, two polynomials
$f,g$ such that $\K[x,y,z]^{\X}=\K[f,g]$.

\subsection{Coordinates}
A polynomial $f\in \K[x_1,\ldots,x_n]$ is called a {\it
coordinate} if there exists a list of polynomials $f_1,\ldots,
f_{n-1}$ such that $\K[x_1,\ldots, x_n]=\K[f,f_1,\ldots,
f_{n-1}]$. In the same way a list $f_1,\ldots, f_r$ of
polynomials, with $r\leq n$, is called a {\it system of
coordinates} if there exists a list $f_{r+1},\ldots,f_n$ of
polynomials such that $\K[x_1,\ldots,x_n]=\K[f_1,\ldots,f_n]$. A
system of coordinates of length $n$ is called a {\it coordinate
system} of $\K[x_1,\ldots,x_n]$.

The famous Abhyankar-Moh theorem \cite{abhyankar-moh75a} states
that a polynomial $f$ in $\K[x,y]$ is a coordinate if and only if
$\K[x,y]/f$ is $\K$-isomorphic to $\K^{[1]}$. In the case of three
variables we have the following result proved by Kaliman in
\cite{kaliman2002a} for the case $\K=\C$ and extended to the
general case in \cite{daigle-kaliman2004a} by using Kambayashi's
transfer principle \cite{kambayashi75a}.
\begin{theorem}\label{kaliman-daigle}Let $f$ be a polynomial in
$\K[x,y,z]$ and assume that for all but finitely many $\alpha\in
\K$ the $\K$-algebra $\K[x,y,z]/(f-\alpha)$ is $\K$-isomorphic to
$\K^{[2]}$. Then $f$ is a coordinate of $\K[x,y,z]$.
\end{theorem}

As a consequence of theorem \ref{kaliman-daigle}, if a polynomial
$f$ is such that $\K(f)[x,y,z]$ is $\K(f)$-isomorphic to
$\K(f)^{[2]}$ then $f$ is a coordinate
\cite{maubach2003a,elkahoui2005a}. This is the version of
Kaliman's theorem we will use in this paper. However, it is not
clear how to compute polynomials $g,h$ such that $f,g,h$ is a
coordinate system of $\K[x,y,z]$ since the original proof given in
\cite{kaliman2002a} is of topological nature.

\subsection{Rank of a derivation}
Let $\X$ be a $\K$-derivation of $\K[\underline{x}]=\K[x_1,\ldots,
x_n]$. As defined in \cite{freudenburg95a} the co-rank of $\X$,
denoted by $corank(\X)$, is the unique nonnegative integer $r$
such that $\K[\underline{x}]^{\X}$ contains a system of
coordinates  of length $r$ and no system of coordinates of length
greater than $r$. The rank of $\X$, denoted by $rank(\X)$, is
defined by $rank(\X)=n-corank(\X)$. Intuitively, the rank of $\X$
is the minimal number of partial derivatives needed to express
$\X$. The only one derivation of rank $0$ is the zero derivation.
Any $\K$-derivation of rank $1$ is of the form
$p(f_1,\ldots,f_{n-1}){\partial}_{f_n}$, where $f_1,\ldots,f_n$ is
a coordinate system. Such a derivation is locally nilpotent if and
only if $p$ does not depend on $f_n$.

Let $\X$ be a nonzero locally nilpotent $\K$-derivation of
$\K[x_1,\ldots,x_n]$, write $\X=\sum a_i{\partial}_{x_i}$ and set
$c=\gcd(a_1,\ldots,a_n)$. We say that $\X$ is irreducible if $c$
is a constant of $\K^{\star}$. It is well known that $\X(c)=0$ and
$\X=c\Y$ where $\Y$ is an irreducible locally nilpotent
$\K$-derivation. Moreover, this decomposition is unique up to a
unit, i.e., if $\X=c_1\Y_1$, where $\Y_1$ is irreducible, then
there exists an nonzero constant $\mu\in \K^{\star}$ such that
$c_1=\mu c$ and $\Y=\mu\Y_1$. Given any irreducible locally
nilpotent $\K$-derivation of $\K[x_1,\ldots, x_n]$ and any $c$
such that $\X(c)=0$, the derivations $\X$ and $c\X$ have the same
rank. Thus, for rank computation we may reduce, without loss of
generality, to irreducible derivations.

\subsection{The plinth ideal of a locally nilpotent derivation}

Let $\A$ be a ring, $\X$ be a locally nilpotent derivation of $\A$
and let
$$\S^{\X}:=\{\X(a)\; \vert\; \X^{2}(a)=0\}.$$ Clearly $\S^{\X}$ is an
ideal of $\A^{\X}$, called the {\it plinth ideal} of $\X$. In case
$\A=\K[x,y,z]$ we have the following result from
\cite{daigle-kaliman2004a} which is a consequence of faithful
flatness of $\K[x,y,z]$ as $\K[x,y,z]^{\X}$-module.
\begin{theorem}\label{bonnet}Let $\X$ be a locally nilpotent
$\K$-derivation of $\K[x,y,z]$. Then the ideal $\S^{\X}$ is
principal.
\end{theorem}

As we will see in the sequel, the ideal $\S^{\X}$ contains a
crucial information for computing the rank of a locally nilpotent
derivation in dimension $3$. Before going on to the details on how
to exploit this information for our purpose we first focus on how
to compute a generator of this ideal.

\section{Minimal local slices}\label{sec:local_slice}In this
section we give an algorithm to compute a generator of the ideal
$\S^{\X}$ in case $\X$ is a locally nilpotent $\K$-derivation of
$\K[x,y,z]$. Besides theorem \ref{bonnet}, our algorithm strongly
depends on the fact that $\K[x,y,z]^{\X}$ is generated by $2$
polynomials. Since we do not have at disposal an algorithmic
version of Miyanishi theorem we assume a generating system of
$\K[x,y,z]^{\X}$ to be available.

\begin{definition}Let $\A$ be a domain and $\mathcal{X}$ be a locally
nilpotent derivation of $\A$. A local slice $s$ of $\mathcal{X}$
is called minimal if for any local slice $v$ such that
$\mathcal{X}(v)\;\vert\; \mathcal{X}(s)$ we have
$\mathcal{X}(v)=\mu \mathcal{X}(s)$, where $\mu$ is a unit of
$\A$.
\end{definition}

\begin{lemma}\label{reduction_lemma}Let $\A$ be a domain and $\X$ be
a locally nilpotent derivation of $\A$. Let $s$ be a local slice
of $\X$, $p$ be a factor of $\X(s)=c$ and write $c=pc_1$. Then
there exists $s_1\in \A$ such that $\X(s_1)=c_1$ if and only if
the ideal $p\A$ contains an element of the form $s+a$ where $a\in
\A^{\X}$.
\end{lemma}
\begin{proof}``$\Rightarrow$") Assume that there exists a local slice
$s_1$ of $\X$ such that $\X(s_1)=c_1$. Then $\X(ps_1-s)=0$ and so
$ps_1-s=a$ where $a$ is a constant of $\X$. This proves that $p\A$
contains $s+a$.

``$\Leftarrow$") Assume now that the ideal $p\A$ contains an
element of the form $s+a$, where $a$ is a constant of $\X$, and
write $s+a=ps_1$. Then $\X(s)=p\X(s_1)$ and so $\X(s_1)=c_1$.
\end{proof}

\begin{proposition}\label{local_slice}Let $\A$ be a UFD, $\X$ be a
locally nilpotent derivation of $\A$ and $s$ be a local slice of
$\X$. Then the following hold:

i) there exists a minimal local slice $s_0$ of $\X$ such that
$\X(s_0)\;\vert\; \X(s)$,

ii) in case $\S^{\X}$ is a principal ideal, it is generated by
$\X(s)$ for any minimal local slice $s$ of $\X$.
\end{proposition}
\begin{proof}$i)$ Let $s$ be a local slice of $\X$ and write
$\X(s)=\mu p_1^{m_1}\cdots p_r^{m_r}$, where $\mu$ is a unit and
the $p_i$'s are primes, and set $m=\sum_im_i$. We will prove the
result by induction on $m$.

For $m=0$ we have $\X(s)=\mu$, and so $\mu^{-1}s$ is a slice of
$\X$. This shows that $s$ is a minimal local slice of $\X$. Let us
now assume the result to hold for $m-1$ and let $s$ be a local
slice of $\X$, with $\X(s)=\mu p_1^{m_1}\cdots p_r^{m_r}$ and
$\sum_im_i=m$. Then we have one of the following cases.

- For any $i=1,\ldots,r$ the ideal $p_i\A$ does not contain any
element of the form $s+a$ with $\X(a)=0$. In this case $s$ is a
minimal local slice of $\X$ by lemma \ref{reduction_lemma}.

- There exists $i$ such that $p_i\A$ contains an element of the
form $s+a$, with $\X(a)=0$. Without loss of generality we may
assume that $i=1$. If we write $s+a=p_1s_1$ then
$\X(s_1)=p_1^{m_1-1}p_2^{m_2}\cdots p_r^{m_r}$, and by using
induction hypothesis we get a minimal local slice $s_0$ of $\X$
such that $\X(s_0)\;\vert\; \X(s_1)$. Since $\X(s_1)\;\vert
\;\X(s)$ we get the result in this case.

$ii)$ Assume now that $\S^{\X}$ is principal and let $c$ be a
generator of this ideal, with $c=\X(s_0)$. Let $s$ be a minimal
local slice of $\X$. Since $\X(s)\in \S^{\X}$ we may write
$\X(s)=\mu\X(s_0)$. The fact that $s$ is minimal implies that
$\mu$ is a unit of $\A^{\X}$, and so $\X(s)$ generates $\S^{\X}$.
\end{proof}

The main question to be addressed, if we want to have an
algorithmic version of proposition \ref{local_slice}, is to check,
for a given prime $p$ of $\A$, whether $p\A\cap \A^{\X}[s]$
contains a monic polynomial of degree $1$ with respect to $s$. In
case $\A$ is an affine ring over a computable field $\K$ this
problem may be solved by using Gr\"obner bases theory, see e.g.,
\cite{adams-loustaunau94a,becker-weispfenning93a,cox-little-oshea98a}.
We only deal here with the case where $\A$ and $\A^{\X}$ are
polynomial rings over a field since this fits our need.

\begin{proposition}\label{grobner_elimination}Let $\ci$ be an ideal of
$\K[x_1,\ldots,x_n]=\K[\underline{x}]$ and
$\underline{h}=h_1,\ldots, h_t$ be a list of algebraically
independent polynomials of $\K[\underline{x}]$. Let
$\underline{u}=u_1,\ldots,u_t$ be a list of new variables and
$\cj$ be the ideal of $\K[\underline{u},\underline{x}]$ generated
by $\I$ and $h_1-u_1,\ldots, h_t-u_t$. Let $G$ be a Gr\"obner
basis of $\cj$ with respect to the lexicographic order $u_1\prec
\cdots \prec u_t\prec x_1\prec \cdots \prec x_n$, and
$\{g_1,\ldots, g_v\}=G\cap \K[\underline{u}]$. Then:

i)  $\{g_1,\ldots,g_v\}$ is a Gr\"obner basis of $\cj\cap
\K[\underline{u}]$ with respect to the lexicographic order
$u_1\prec \cdots \prec u_t$,

ii) the $\K$-isomorphism $u_i\in \K[\underline{u}]\longmapsto
h_i\in \K[\underline{h}]$ maps $\cj\cap\K[\underline{u}]$ onto
$\ci\cap\K[\underline{h}]$.
\end{proposition}

In our case, we have $\ci=p\K[x,y,z]$ for some polynomial $p$, and
$\K[\underline{h}]=\K[f,g,s]$ where $f,g$ is a generating system
of $\K[x,y,z]^{\X}$ and $s$ is a local slice of $\X$. Let
$u_1,u_2,u_3$ be new variables and $\cj$ be the ideal of
$\K[u_1,u_2,u_3,x,y,z]$ generated by $p,f-u_1,g-u_2,s-u_3$. Let
$G$ be a Gr\"obner basis of $\cj$ with respect to the
lexicographic order $u_1\prec u_2\prec u_3\prec x\prec y\prec z$
and $G_1=G\cap \K[u_1,u_2,u_3]$. According to proposition
\ref{grobner_elimination}, the ideal $p\K[x,y,z]\cap \K[f,g,s]$
contains a polynomial of the form $s+a(f,g)$ if and only if $G_1$
contains a monic polynomial $\ell(u_1,u_2,u_3)$ of degree $1$ with
respect to $u_3$. In this case the polynomial we are looking for
is $\ell(f,g,s)$.

The following algorithm gives the main steps to be performed for
computing a minimal local slice of a given locally nilpotent
derivation in dimension $3$.

\begin{algorithm}[ht]
\caption{: Minimal local slice algorithm.}
\medskip\noindent
\begin{algorithmic}[1]
\label{mini_local_slice} \REQUIRE A locally nilpotent
$\K$-derivation $\X$ of
    $\K[x,y,z]$ and a generating system $f,g$ of $\K[x,y,z]^{\X}$.

\ENSURE A minimal local slice $s$ of $\X$.

\mni\STATE Compute a local slice $s_0$ of $\X$.

\STATE Write $\X(s_0)=p_1^{m_1}\cdots p_r^{m_r}$, where the
$p_i$'s are primes.

\STATE $s:=s_0$.

\FOR {$i$ from $1$ to $r$}

\FOR{$j$ from $1$ to $m_i$}

\STATE Let $G$ be a Gr\"obner basis of
$\ci(p_i,f-u_1,g-u_2,s-u_3)$ with respect to the lex-order
$u_1\prec u_2\prec u_3\prec x\prec y\prec z$, and let $G_1=G\cap
\K[u_1,u_2,u_3]$.

\IF {$G_1$ contains a monic polynomial of degree $1$ with respect
to $u_3$, say $u_3+a(u_1,u_2)$}

\STATE Write $s+a(f,g)=p_is_1$.

\STATE $s:=s_1$.

\ELSE \STATE Break.

\ENDIF

\ENDFOR

\ENDFOR
\end{algorithmic}
\end{algorithm}

\section{Computing the rank in dimension three}
\label{sec:rank}An irreducible locally nilpotent $\K$-derivation
$\X$ of $\K[x_1,\ldots,x_n]$ is of rank $1$ if and only if
$\K[x_1,\ldots,x_n]^{\X}=\K^{[n-1]}$ and $\X$ has a slice, see
\cite{freudenburg95a}. In dimension $3$, and taking into account
theorem \ref{miyanishi}, an irreducible locally nilpotent
derivation is of rank one if and only if Algorithm
\ref{mini_local_slice} produces a slice. Therefore, we only need
to characterize derivations of rank two.

\begin{theorem}\label{rank_two_theorem}Let $\X$ be an irreducible
locally nilpotent derivation of $\K[x,y,z]$ and assume that
$rank(\X)\neq 1$. Let us write $\K[x,y,z]^{\X}=\K[f,g]$ and
$\S^{\X}=c\K[f,g]$. Then the following are equivalent:

i) $rank(\X)=2$,

ii) $c=\ell(u)$, where $\ell$ is a univariate polynomial and $u$
is a coordinate of $\K[f,g]$,

iii) $c=\ell(u)$, where $u$ is a coordinate of $\K[x,y,z]$.
\end{theorem}
\begin{proof}$i)\Rightarrow ii)$ Assume that $rank(\X)=2$ and
let $u,v,w$ be a coordinate system such that $\X(u)=0$. The
$\K$-derivation $\X$ is therefore a $\K[u]$-derivation of
$\K[u][v,w]$, and since $\K[u]$ is a UFD there exists $p\in
\K[x,y,z]$ such that $\K[f,g]=\K[u,p]$. This proves that $u$ is a
coordinate of $\K[f,g]$. Let us now view $\X$ as
$\K(u)$-derivation of $\K(u)[v,w]$. Since $\X$ is irreducible and
according to theorem \ref{rentshcler} there exists
$s=\frac{h(u,v,w)}{k(u)}$ such that $\X(s)=1$, and so
$\X(h)=k(u)$. Let $c$ be a generator of $\S^{\X}$. Then
$c\;\vert\; k(u)$, and since $\K[u]$ is factorially closed in
$\K[u,v,w]$ we have $c=\ell(u)$ for some univariate polynomial
$\ell$.

$ii)\Rightarrow iii)$ Assume that $c=\ell(u)$, where $u$ is a
coordinate of $\K[f,g]$ and write $\K[f,g]=\K[u,p]$. Let $s$ be
such that $\mathcal{X}(s)=c$. If we view $\X$ as
$\K(u)$-derivation of $\K(u)[x,y,z]$ then
$\K(u)[x,y,z]^{\X}=\K(u)[p]$ and $\X(c^{-1}s)=1$. By applying
lemma \ref{wright} we get $\K(u)[x,y,z]=\K(u)[p,s]$. From theorem
\ref{kaliman-daigle} we deduce that $u$ is a coordinate of
$\K[x,y,z]$.

$iii)\Rightarrow i)$ We have $\X(c)=\ell^{\prime}(u)\X(u)=0$, and
so $\X(u)=0$. On the other hand, since $u$ is assumed to be a
coordinate of $\K[x,y,z]$ we have $rank(\X)\leq 2$. By assumption
we have $rank(\X)\neq 1$ and so $rank(\X)=2$.
\end{proof}

The conditions of $ii)$ in theorem \ref{rank_two_theorem} are in
fact algorithmic. Indeed, it is algorithmically possible to check
whether a given polynomial in two variables is a coordinate, see
e.g.
\cite{abhyankar-moh75a,berson-essen2000a,elkahoui2004c,shpilrain97a}.
We will use the algorithm given in \cite{elkahoui2004c}, but it is
worth mentioning that from the complexity point of view the
algorithm given in \cite{shpilrain97a} is the most efficient as
reported in \cite{shpilrain2005a}. On the other hand, condition
$c=\ell(u)$ may be checked by using a special case, called {\it
uni-multivariate decomposition}, of functional decomposition of
polynomials, see e.g., \cite{gathen90}. It is important to notice
here that uni-multivariate decomposition is essentially unique.
Namely, if $c=\ell(u)=\ell_1(u_1)$, where $u$ and $u_1$, are
undecomposable, then there exist $\mu\in \K^{\star}$ and $\nu\in
\K$ such that $u_1=\mu u+\nu$. Due to the particular nature of our
decomposition problem it seems more convenient to use the
following proposition.

\begin{proposition}\label{decomposition_lemma}Let $c(\underline{x})\in
\K[x_1,\ldots,x_n]$ be nonconstant and
$\underline{u}=u_1,\ldots,u_n$ be a list of new variables. Then
the following are equivalent:

i) $c(\underline{x})=\ell(y_1(\underline{x}))$, where $\ell$ is a
univariate polynomial and $y_1$ is a coordinate of
$\K[x_1,\ldots,x_n]$,

ii) $y_1(\underline{x})-y_1(\underline{u})\;\vert\;
c(\underline{x})-c(\underline{u})$ and $y_1$ is a coordinate of
$\K[x_1,\ldots,x_n]$.
\end{proposition}
\begin{proof}$i)\Rightarrow ii)$ Let $\ell$ be a univariate polynomial
and $t,t^{\prime}$ be new variables. Then we have
$t-t^{\prime}\;\vert\; \ell(t)-\ell(t^{\prime})$. This shows that
$y_1(\underline{x})-y_1(\underline{u})\;\vert\;
\ell(y(\underline{x}))-\ell(y(\underline{u}))$.

$ii)\Rightarrow i)$ Let $y_2,\ldots,y_n$ be polynomials such that
$\underline{y}=y_1,\ldots,y_n$ is a coordinate system of
$\K[\underline{x}]$, and let $v_i=y_i(\underline{u})$. Then
$\underline{v}=v_1,\ldots,v_n$ is a coordinate system of
$\K[\underline{u}]$. Let us write
$c(\underline{x})-c(\underline{u})=(y_1(\underline{x})-
y_1(\underline{u}))A(\underline{u},\underline{x})$ and
$c(\underline{x})=\ell(\underline{y})$. Then we have
\begin{equation}\label{division_equation}\ell(\underline{y})-
\ell(\underline{v})=(y_1-v_1)B(\underline{v},\underline{y}).
\end{equation}
Let us now write
$\ell(\underline{y})=\sum_{\alpha}a_{\alpha}(y_1)y_2^{\alpha_2}
\cdots y_n^{\alpha_n}$, where $\alpha=(\alpha_2,\ldots,\alpha_n)$.
After substituting $y_1$ to $v_1$ in the relation
(\ref{division_equation}) we get
$$\sum_{\alpha}a_{\alpha}(y_1)y_2^{\alpha_2}\cdots y_n^{\alpha_n}
-\sum_{\alpha}a_{\alpha}(y_1)v_2^{\alpha_2}\cdots
v_n^{\alpha_n}=0.$$ Taking into account the fact that
$v_2,\ldots,v_n,y_2,\ldots,y_n$ are algebraically independent over
$\K[y_1]$ we get $a_{\alpha}=0$ for any $\alpha\neq 0$. This
proves that $\ell(\underline{y})$ is a polynomial in terms of
$y_1$.
\end{proof}

\begin{algorithm}[ht]
\caption{: Rank algorithm.}
\medskip\noindent
\begin{algorithmic}[1]\label{rank_algorithm}
\REQUIRE A locally nilpotent $\K$-derivation $\X$ of
    $\K[x,y,z]$ and a generating system $f,g$ of $\K[x,y,z]^{\X}$.

\ENSURE The rank of $\X$.

\mni \STATE Write
$\X=a_1{\partial}_{x}+a_2{\partial}_y+a_3{\partial}_z$. Let
$c_1=\gcd(a_1,a_2,a_3)$ and $\X=c_1\Y$.

\STATE By Algorithm \ref{mini_local_slice}, compute a minimal
local slice $s$ of $\Y$ and let $c=\Y(s)$.

\IF{$c$ is a unit}

\STATE $rank(\X)=1$.

\ELSE

\STATE Compute a factorization of $c(f,g)-c(t_1,t_2)$ in
$\K[f,g,t_1,t_2]$.

\IF{no factor of $c(f,g)-c(t_1,t_2)$ is of the form
$u(f,g)-u(t_1,t_2)$}

\STATE $rank(\X)=3$.

\ELSE

\IF{$u$ is a coordinate of $\K[f,g]$ for a factor of the form
$u(f,g)-u(t_1,t_2)$ of $c(f,g)-c(t_1,t_2)$}

\STATE $rank(\X)=2$.

\ELSE

\STATE $rank(\X)=3$.

\ENDIF

\ENDIF

\ENDIF
\end{algorithmic}
\end{algorithm}

\mni As defined in \cite{daigle96}, a derivation $\X$ of rank $r$
of $\K[\underline{x}]$ is rigid if for any coordinate systems
$y_1,\ldots,y_n$ and $z_1,\ldots,z_n$ such that
$\K[y_1,\ldots,y_{n-r}]\subseteq \K[\underline{x}]^{\X}$ and
$\K[z_1,\ldots,z_{n-r}]\subseteq \K[\underline{x}]^{\X}$ we have
$\K[y_1,\ldots,y_{n-r}]=\K[z_1,\ldots,z_{n-r}]$. The main results
of \cite{daigle96} lie behind the fact that locally nilpotent
derivations in dimension $3$ are rigid. In fact only the rank two
case is nontrivial since in general derivations of rank $0, 1$ and
$n$ are obviously rigid. The characterization $ii)$ of rank two
derivations given in theorem \ref{rank_two_theorem} gives a more
precise information. Indeed, it tells that if a coordinate of
$\K[x,y,z]$ belongs to $\K[x,y,z]^{\X}$ then it may be found by
decomposing the generator of the plinth ideal $\S^{\X}$. The fact
that rank two derivations are rigid is then an obvious consequence
of the uniqueness property of uni-multivariate decomposition.

\mni In case a given derivation $\X$ is of rank $2$, Algorithm
\ref{rank_algorithm} does not produce a coordinate system $u,v,w$
such that $\X(u)=0$. Computing such a coordinate system, which is
equivalent to obtaining an algorithmic version of theorem
\ref{kaliman-daigle}, is a necessary step towards our ultimate
goal, namely algorithmically classifying locally nilpotent
derivations in dimension $3$.

\section{Comments on implementation}\label{sec:comments}
Before implementing algorithms for locally nilpotent derivations
of $\K[x,y,z]$ we must first specify how such objects are to be
concretely represented. Any chosen representation should address
the two following problems.
\begin{enumerate}
    \item {\it Recognition problem:} Given a $\K$-derivation $\X$ of
    $\K[x,y,z]$, check whether $\X$ is locally nilpotent.
    \item {\it Kernel problem:} Given a locally nilpotent
    $\K$-derivation $\X$ of $\K[x,y,z]$, compute $f,g$ such that
    $\K[x,y,z]^{\X}=\K[f,g]$.
\end{enumerate}

Usually, a derivation $\X$ of $\K[x,y,z]$ is written as a
$\K[x,y,z]$-linear combination of the partial derivatives
$\partial_{x},\partial_y,\partial_z$. However, with such a
representation, the recognition and kernel problems are nowhere
near completely solved. To our knowledge, only the weighted
homogeneous case of the recognition problem is solved, see
\cite{essen2003a}. One way to go round this hurdle is to opt for
another representation.

The Jacobian representation gives another alternative to represent
locally nilpotent derivations. Indeed, any irreducible locally
nilpotent $\K$-derivation of $\K[x,y,z]$ is, up to a nonzero
constant in $\K$, equal to $\jac(f,g,.)$, see \cite{daigle97b}. In
order to check whether a Jacobian derivation $\X=\jac(f,g,.)$ is
locally nilpotent it suffices to check that
$\X^{d+1}(x)=\X^{d+1}(y)=\X^{d+1}(z)=0$, where $d=
\deg(f)\deg(g)$, see \cite{essen_book}. However, it is still not
clear how such a representation could help in solving the kernel
problem. Nevertheless, we may always check whether this ring of
constants is generated over $\K$ by $f,g$ by using van den Essen's
kernel algorithm \cite{essen_book}.

Due to the above discussed issues, we have restricted our
implementation to the case of derivations of $\K[x,y,z]$
represented in a Jacobian form, say $\jac(f,g,.)$, and whose ring
of constants is generated by $f,g$. The computer Algebra system we
used for implementation is \Maple 9.


\end{document}